\pdfoutput=1 
\documentclass[11pt]{amsart}
\usepackage[utf8]{inputenc}
\usepackage{amssymb,amsmath,amsthm,enumerate,enumitem,colonequals,mlmodern,tikz-cd,microtype}
\usepackage[cal=euler,bfcal,bb=px,bfbb]{mathalpha}
\usepackage[top=3.75cm, bottom=3cm, left=3.5cm, right=3.5cm]{geometry}

\makeatletter
\@namedef{subjclassname@2020}{\textup{2020} Mathematics Subject Classification}
\makeatother

\usepackage{xcolor}
\colorlet{darkblue}{blue!55!black}
\colorlet{darkcyan}{cyan!50!black}
\colorlet{darkgreen}{green!60!black}

\usepackage{hyperref}
\hypersetup{
    colorlinks=true,
    linkcolor= darkblue,
    urlcolor= darkcyan,
    citecolor= darkgreen,
}
 
\def\eqref#1{\textcolor{darkblue}{(\ref{#1})}}

\PassOptionsToPackage{hyphens}{url}
\usepackage{hyperref}

\usepackage[nameinlink]{cleveref} 
\Crefformat{section}{#2\S#1#3} 
\Crefmultiformat{section}{#2\S\S#1#3}{ and~#2#1#3}{, #2#1#3}{, and~#2#1#3}
\crefname{hypothesis}{hypothesis}{hypotheses}
\Crefname{hypothesis}{Hypothesis}{Hypotheses}

\usepackage[pagewise]{lineno}
\let\oldequation\equation
\let\oldendequation\endequation
\renewenvironment{equation}{\linenomathNonumbers\oldequation}{\oldendequation\endlinenomath}
\expandafter\let\expandafter\oldequationstar\csname equation*\endcsname
\expandafter\let\expandafter\oldendequationstar\csname endequation*\endcsname
\renewenvironment{equation*}{\linenomathNonumbers\oldequationstar}{\oldendequationstar\endlinenomath}
\let\oldalign\align
\let\oldendalign\endalign
\renewenvironment{align}{\linenomathNonumbers\oldalign}{\oldendalign\endlinenomath}
\expandafter\let\expandafter\oldalignstar\csname align*\endcsname
\expandafter\let\expandafter\oldendalignstar\csname endalign*\endcsname
\renewenvironment{align*}{\linenomathNonumbers\oldalignstar}{\oldendalignstar\endlinenomath}


\newcounter{intro}
\newcounter{HypCounter}

\newtheorem{introthm}[intro]{Theorem}

\newtheorem{introprop}[intro]{Proposition}

\theoremstyle{plain}
\newtheorem{theorem}{Theorem}[section]
\newtheorem{lemma}[theorem]{Lemma}

\newtheorem{proposition}[theorem]{Proposition}

\theoremstyle{definition}

\newtheorem{example}[theorem]{Example}

\newtheorem*{hypothesis*}{Hypothesis}

\newtheorem{remark}[theorem]{Remark}

\newtheorem*{ack}{Acknowledgements}

\numberwithin{equation}{section}
\numberwithin{theorem}{section}

\title[Integral transforms and singularity categories]{Integral transforms \\ on singularity categories \\ for Noetherian schemes}

\author[U. ~Dutta]{Uttaran Dutta}
\address{U. ~Dutta, 
Department of Mathematics,
University of South Carolina, 
Columbia, SC 29208,
U.S.A.}
\email{udutta@email.sc.edu}

\author[P. ~Lank]{Pat Lank}
\address{P.~Lank,
Dipartimento di Matematica “F. Enriques”, Universit\`{a} degli Studi di Milano, Via Cesare
Saldini 50, 20133 Milano, Italy}
\email{plankmathematics@gmail.com}
 
\author[K. ~Manali-Rahul]{Kabeer Manali-Rahul}
\address{K. ~Manali-Rahul,
Center for Mathematics and its Applications, 
Mathematical Science Institute, Building 145, 
The Australian National University, 
Canberra, ACT 2601, Australia}
\email{kabeer.manalirahul@anu.edu.au}

\date{\today}



\begin{document}
    
\begin{abstract}
    This work studies conditions under which integral transforms induce exact functors on singularity categories between schemes that are proper over a Noetherian base scheme. A complete characterization for this behavior is provided, which extends earlier work of Ballard and Rizzardo. We leverage a description of the bounded derived category of coherent sheaves as finite cohomological functors on the category of perfect complexes, which is an application of Neeman's approximable triangulated categories, to reduce arguments to an affine local setting. Moreover, we study adjoints of such functors, extend a result of Olander to varieties with mild singularities, and provide an obstruction for derived equivalences between singular varieties.
\end{abstract}

\maketitle

\section{Introduction}
\label{sec:intro}

Our work characterizes the conditions under which integral transforms induce exact functors on the singularity categories for schemes of interest. Specifically, we work with schemes $Y_i$ that are proper over a Noetherian base, and study when an integral transform induce functors on the category of perfect complexes and/or bounded derived category of coherent sheaves. These are respectively denoted $\operatorname{Perf}(Y_i)$ and $D^b_{\operatorname{coh}}(Y_i)$. 

This not only builds upon the prior work of \cite{Ballard:2009,Rizzardo:2017}, but also gives fresh insight into the homological information contained in derived categories for singular varieties. This study has been hitherto mostly been done only in the smooth setting. So to set the stage, we start with some history. 

Integral transforms were first introduced in \cite{Mukai:1981} for Abelian varieties, but have been studied in general settings since than. An important result, due to Orlov \cite{Orlov:1997}, states that a fully faithful functor $F\colon D^b_{\operatorname{coh}}(Y_1) \to D^b_{\operatorname{coh}}(Y_2)$ between \textit{smooth} projective complex varieties is naturally isomorphic to an integral transform. This means $F$ is naturally isomorphic to a functor of the form $\mathbf{R}p_{2,\ast} (\mathbf{L}p_1^\ast (-) \otimes^\mathbf{L} K)$ for $K$ an object of $D_{\operatorname{qc}}(Y_1\times_k Y_2)$ and $p_i\colon Y_1\times_k Y_2 \to Y_i$ the projection morphisms.

There are a few reasons for studying integral transforms between varieties, and more generally, schemes. One such is to extract geometric information about varieties and their derived categories. However, until fairly recently, most attention has been directed towards smooth varieties. Luckily, progress has been made on extending classical results such as \cite{Orlov:1997} to the singular context (e.g.\ \cite{Ballard:2009, Ruiperez/Martin/deSalas:2007, Ruiperez/Martin/deSalas:2009, Lunts/Orlov:2010}). Moreover, integral transforms on singular projective curves has enjoyed progress as well \cite{Burban/Kreussler:2006, Burban/Zheglov:2018, LopezMartin:2014, Spence:2023}. 

This largely motivated our work. Specifically, it is the initial stage for a program on understanding derived categories of singular varieties through integral transforms. A key first step is to understand when integral transforms induce functors on categories of interest. We push this to a larger generality by allowing cases that are arithmetically flavored (e.g.\ proper schemes over a Dedekind domain).
 
\subsection{Preservation}
\label{sec:intro_preservation}

The structure of the derived category for a singular variety is much less understood than that of smooth varieties. However, recent efforts have made strides towards bettering our knowledge. These include, for instance, connections to regularity \cite{Neeman:2021a, Clark/Lank/Parker/ManaliRahul:2024}, the closedness of the singular locus of a scheme \cite{Dey/Lank:2024a, Iyengar/Takahashi:2019}, and methods for detecting singularities in birational geometry \cite{Lank/Venkatesh:2024, Lank/McDonald/Venkatesh:2025,BILMP:2023,Lank:2025}.

One approach to hurdling over this difficulty in understanding $D^b_{\operatorname{coh}}(X)$ is to study other categories that can be cooked up from it. Recall that the singularity category of a Noetherian scheme $X$, denoted $D_{\operatorname{sg}}(X)$, is defined as the Verdier localization of $D^b_{\operatorname{coh}}(X)$ by $\operatorname{Perf}(X)$. This notion first appeared in the affine setting \cite{Buchweitz/Appendix:2021} but was later rediscovered later in the geometric context \cite{Orlov:2004}. 

The singularity category detects regularity for a Noetherian scheme, as $D_{\operatorname{sg}}(X)$ is trivial if and only if $X$ is regular. There has been work on the $K$-theoretic aspects of the singularity category \cite{Pavic/Shinder:2021} and on when varieties have triangle equivalent singularity categories \cite{Kalck:2021,Knorrer:1987,Matsui/Takahashi:2017}. Interestingly enough, such (triangulated) equivalences have arisen as induced functors from integral transforms on $D^b_{\operatorname{coh}}$ (see \Cref{sec:prelim_singularity_category}). One goal for us was to characterize more explicitly computable conditions for when integral transforms induce such functors.

This brings us to our first result.

\begin{introthm}[see \Cref{thm:integral_transform_induced_singularity}]\label{introthm:integral_transform_induced_singularity}
    Let $f_1 \colon Y_1 \to S$ and $f_2 \colon Y_2 \to S$ be proper morphisms with $S$ a Noetherian scheme. Consider the fibered square:
    \begin{displaymath}
        \begin{tikzcd}[ampersand replacement=\&]
            {Y_1\times_S Y_2} \& {Y_2} \\
            {Y_1} \& S
            \arrow["{p_2}", from=1-1, to=1-2]
            \arrow["{p_1}"', from=1-1, to=2-1]
            \arrow["\lrcorner"{anchor=center, pos=0.125}, draw=none, from=1-1, to=2-2]
            \arrow["{f_2}", from=1-2, to=2-2]
            \arrow["{f_1}"', from=2-1, to=2-2]
        \end{tikzcd}
    \end{displaymath}
    Then the following are equivalent for any object $K$ in $D^-_{\operatorname{coh}}(Y_1 \times_S Y_2)$:
    \begin{enumerate}
        \item $\Phi_K (\operatorname{Perf}(Y_1))$ and $\Phi_K (D^b_{\operatorname{coh}}(Y_1))$ are respectively contained in $\operatorname{Perf}(Y_2)$ and $D^b_{\operatorname{coh}}(Y_2)$ 
        \item $\mathbf{R}p_{i,\ast} (K\otimes^{\mathbf{L}} \operatorname{Perf}(Y_1\times_S Y_2))\subseteq \operatorname{Perf}(Y_i)$ for each $i$.
    \end{enumerate}
    Consequently, if either of the above conditions hold, then $\Phi_K$ induces an exact functor $\breve{\Phi_K}\colon D_{\operatorname{sg}}(Y_1) \to D_{\operatorname{sg}}(Y_2)$.
\end{introthm}

\Cref{introthm:integral_transform_induced_singularity} was previously known for projective varieties over a field (see \cite[Lemmas 3.17 \& 3.18]{Ballard:2009}), which was later extended to schemes that are projective over a Noetherian base (see \cite[Corollary 5.3]{Rizzardo:2017}). Our result upgrades to `proper' schemes over a Noetherian base. This allows for further study of integral transforms in cases of schemes that might be proper, but not projective over a base, e.g.\ proper varieties over a field which are not projective \cite{Nagata:1958, Schroer:1999, Oda:1988, Kollar:2006}. See \Cref{ex:bergh_schnurer_stacks} for cases when the base scheme is not affine (i.e.\ `relative' settings) or in mixed characteristic settings.

One key ingredient in the proof of \Cref{introthm:integral_transform_induced_singularity} is a particular description of $D^b_{\operatorname{coh}}(-)$ for a Noetherian proper scheme. Specifically, the objects can be viewed as finite cohomological functors on the category of perfect complexes. This allows one to reduce many arguments to the affine situation. The description is an application of Neeman's recent work on approximable triangulated categories \cite{Neeman:2021c}. See \Cref{lem:approx_by_proper_lift_functors} for details. 

\subsection{Some consequences}
\label{sec:intro_some_consequences}

Now we highlight applications of our techniques that were developed to prove \Cref{introthm:integral_transform_induced_singularity}. A next goal is to start using such characterizations to extract geometric information from the derived categories of singular varieties. We start with a natural generalization of \cite[Lemma's 3.22 \& 3.23]{Ballard:2009} to the setting of proper schemes over an affine base. 

\begin{introprop}
    [see \Cref{prop:integral_transform_on_dbcoh_iff_left_adjoint}]
    \label{introprop:integral_transform_on_dbcoh_iff_left_adjoint}
    Let $f_1 \colon Y_1 \to S$ and $f_2 \colon Y_2 \to S$ be proper morphisms with $S$ an affine Noetherian scheme. Consider the fibered square:
    \begin{displaymath}
        \begin{tikzcd}[ampersand replacement=\&]
            {Y_1\times_S Y_2} \& {Y_2} \\
            {Y_1} \& S
            \arrow["{p_2}", from=1-1, to=1-2]
            \arrow["{p_1}"', from=1-1, to=2-1]
            \arrow["\lrcorner"{anchor=center, pos=0.125}, draw=none, from=1-1, to=2-2]
            \arrow["{f_2}", from=1-2, to=2-2]
            \arrow["{f_1}"', from=2-1, to=2-2]
        \end{tikzcd}
    \end{displaymath}
    Then the following are equivalent for any $K\in D^-_{\operatorname{coh}}(Y_1\times_S Y_2)$:
    \begin{enumerate}
        \item $\Phi_K\colon D_{\operatorname{qc}}(Y_1) \to D_{\operatorname{qc}}(Y_2)$ admits a left adjoint
        \item $\Phi_K \colon D^b_{\operatorname{coh}}(Y_1) \to D_{\operatorname{qc}}(Y_2)$ factors through $D^b_{\operatorname{coh}}(Y_2)$. 
    \end{enumerate}
    In such a situation, $\Phi_{K^\prime}$ is the left adjoint to $\Phi_K$ on $D_{\operatorname{qc}}$, 
    where $K^\prime:= \mathbf{R} \operatorname{\mathcal{H}\! \mathit{om}} (K, p^!_1 \mathcal{O}_{Y_1})$. 
\end{introprop}

\Cref{introprop:integral_transform_on_dbcoh_iff_left_adjoint} connects the existence of a left adjoint with being able to preserve objects with bounded and coherent cohomology. This becomes a useful tool for studying integral transforms on singular varieties. Indeed, such adjoints are important for studying fully faithful integral transforms between smooth projective varieties as they allow one to infer bounds on Krull dimension of the varieties \cite{Olander:2023}. 

We push such results to varieties with mild singularities. Recall a variety $X$ over $\mathbb{C}$ is said to have \textit{rational singularities} if $\mathcal{O}_X \xrightarrow{ntrl.} \mathbf{R}f_\ast \mathcal{O}_{\widetilde{X}}$ is an isomorphism for $f\colon \widetilde{X} \to X$ a resolution of singularities. This includes toric varieties, quotient singularities on surfaces, and more.

To the best of our knowledge, the following is new and shows us one way as to how integral transforms can be used to extract geometric information in the singular setting.

\begin{introprop}
    \label{introprop:extending_olander_splinters}
    Let $Y_1$ and $Y_2$ be proper, both with rational singularities, over $\mathbb{C}$. Suppose there is $K\in D^-_{\operatorname{coh}}(Y_1\times_k Y_2)$ such that $\Phi_K\colon D^b_{\operatorname{coh}}(Y_1) \to D^b_{\operatorname{coh}}(Y_2)$ is fully faithful.
    If $\Phi_K(\operatorname{Perf}(Y_1)) \subseteq \operatorname{Perf}(Y_2)$, then $\dim Y_1 \leq \dim Y_2$.
\end{introprop}

\Cref{introprop:extending_olander_splinters} extends \cite{Olander:2023} to a class of varieties with mild singularities. Observe that \Cref{introprop:extending_olander_splinters} tells us $\dim Y_1 = \dim Y_2$ when $\Phi_K$ is an equivalence. Its proof becomes straightforward once the existence of adjoints is sorted from our work in the manuscript as one then piggy backs off \cite[Proposition 6.8]{DeDeyn/Lank/ManaliRahul:2024b}. Moreover, we prove a more general result for varieties over arbitrary fields whose singularities are called \textit{birational derived splinters} (\`{a} la Kov\'{a}cs \cite{Lank:2025}); see \Cref{prop:extending_olander_splinters} for details. 

Lastly we determine a property that cannot be enjoyed by a kernel of an integral transform that induces an equivalence of $D^b_{\operatorname{coh}}$ between singular varieties.

\begin{introprop}
    Let $Y_1,Y_2$ be proper varieties over $\mathbb{C}$. Suppose $Y_1$ and $Y_2$ are Fourier--Mukai partners given by a kernel $K$ in $D^b_{\operatorname{coh}}(Y_1\times_S Y_2)$. If $Y_1$ or $Y_2$ is not smooth, then $K\not\in\operatorname{Perf}(Y_1\times_S Y_2)$.
\end{introprop}

This result is proven more generally for proper schemes over a regular base. It gives a slightly better understanding for kernels of integral transforms inducing triangulated equivalences between schemes that have singularities. We are not aware of many statements that tell us properties of kernels as such for singular varieties. So one can take the result above as first strides towards a consequence of our proposed program. See \Cref{prop:obstruction_to_singular_equivalence} for detail.

\subsection{Notation}
\label{sec:intro_notation}

Let $X$ be a Noetherian scheme. Then, we have the following triangulated categories associated to $X$:
\begin{enumerate}
    \item $D(X):=D(\operatorname{Mod}(X))$ is the derived category of $\mathcal{O}_X$-modules.
    \item $D_{\operatorname{qc}}(X)$ is the (strictly full) subcategory of $D(X)$ consisting of complexes with quasi-coherent cohomology.
    \item $D_{\operatorname{coh}}^b(X)$ is the (strictly full) subcategory of $D(X)$ consisting of complexes having bounded and coherent cohomology.
    \item $\operatorname{Perf}(X)$ is the (strictly full) subcategory of $D_{\operatorname{qc}}(X)$ consisting of the perfect complexes on $X$.
\end{enumerate}
If $X$ is affine, then we might at times abuse notation and write $D(R):=D_{\operatorname{qc}}(X)$ where $R:=H^0(X,\mathcal{O}_X)$ is the ring of global sections; similar conventions will occur for the other categories.

\begin{ack}
    Pat Lank and Kabeer Manali-Rahul were both supported under the ERC Advanced Grant 101095900-TriCatApp. Manali-Rahul is a recipient of an Australian Government Research Training Program Scholarship, and would like to thank Universit\`{a} degli Studi di Milano for their hospitality during his stay there. Uttaran Dutta was supported under the National Science Foundation under Award No. 2302263. The authors would like to thank Timothy De Deyn and Nebojsa Pavic for discussions, as well as the anonymous referee for suggestions and improvements.
\end{ack}

\section{Preliminaries}
\label{sec:prelim}

This section gives a (very) brisk recap of the background material. The expert is encouraged to skip to the next section. We fix a Noetherian scheme $X$.

\subsection{Generation for triangulated categories}
\label{sec:prelim_generation}

Let $\mathcal{T}$ be a triangulated category with shift functor $[1]\colon \mathcal{T} \to \mathcal{T}$. We pull content from \cite{Bondal/VandenBergh:2003,Rouquier:2008}. Fix a subcategory $\mathcal{S}$ of $\mathcal{T}$. We denote by $\operatorname{add}(\mathcal{S})$ the strictly full\footnote{This means being closed under isomorphisms.} subcategory of $\mathcal{T}$ consisting of all direct summands of objects of the form $\oplus_{n\in \mathbb{Z}} S_n^{\oplus r_n}[n]$ where $S_n$ belongs to $\mathcal{S}$ for all $n$ and finitely many $r_n$ are nonzero. We inductively define the following subcategories: 
\begin{displaymath}
    \langle\mathcal{S}\rangle_n :=
    \begin{cases}
        \operatorname{add}(\varnothing) & \text{if }n=0, \\
        \operatorname{add}(\mathcal{S} )& \text{if }n=1, \\
        \operatorname{add}(\{ \operatorname{cone}\phi \mid \phi \in \operatorname{Hom}(\langle \mathcal{S}  \rangle_{n-1}, \langle \mathcal{S}  \rangle_1) \}) & \text{if }n>1.
    \end{cases}
\end{displaymath}
Furthermore, set
\begin{displaymath}
\langle\mathcal{S} \rangle:=\langle\mathcal{S} \rangle_\infty:=\bigcup_{n\geq 0} \langle\mathcal{S} \rangle_n.
\end{displaymath}
This is the smallest \textbf{thick} subcategory containing $\mathcal{S}$ (i.e.\ triangulated subcategory closed under direct summands). If $\mathcal{S}$ consist of a single object $G$, then we write the constructions above as $\langle G \rangle_n$ and $\langle G \rangle$. An object $G\in \mathcal{T}$ is called a \textbf{classical} (resp.\ \textbf{strong}) \textbf{generator} for $\mathcal{T}$ if $\langle G \rangle = \mathcal{T}$ (resp.\ $\langle G \rangle_{n} = \mathcal{T}$ for some $n\geq 0$). The \textbf{Rouquier dimension} of $\mathcal{T}$, denoted $\dim \mathcal{T}$, is the smallest $n\geq 0$ such that $\langle G \rangle_{n+1} = \mathcal{T}$ for some object $G$. Similarly, one defines the \textbf{countable Rouquier dimension} of $\mathcal{T}$, denoted by $\operatorname{Cdim}(\mathcal{T})$, as the smallest $n\geq 0$ such that $\mathcal{T} = \langle \mathcal{C}\rangle_{n+1}$ for $\mathcal{C}$ a countable collection of objects in $\mathcal{T}$.

\subsection{Perfect \& pseudocoherent complexes}
\label{sec:prelim_perfectness_pseudocoherence}

Recall an object in $D_{\operatorname{qc}}(X)$ is called \textbf{perfect} if it is locally quasi-isomorphic to a bounded complex of vector bundles. The full subcategory of such objects in $D_{\operatorname{qc}}(X)$ is denoted by $\operatorname{Perf}(X)$. One can check that $\operatorname{Perf}(X)$ is a triangulated subcategory of $D_{\operatorname{qc}}(X)$. It is connected to important geometric properties of a scheme. For example, the structure of this category can be used to detect whether $X$ is regular (e.g.\ \cite{DeDeyn/Lank/ManaliRahul:2024a, Neeman:2022, Clark/Lank/Parker/ManaliRahul:2024, Neeman:2021a}).

An object $E$ in $D_{\operatorname{qc}}(X)$ is called \textbf{pseudocoherent} if for each integer $N$ there is a map $P \to E$ from a perfect complex such that the induced map on cohomology sheaves $\mathcal{H}^n (P) \to \mathcal{H}^n (E)$ is an isomorphism for $n>N$ and is surjective for $n=N$. It follows from \cite[\href{https://stacks.math.columbia.edu/tag/08E8}{Tag 08E8}]{StacksProject} that an object $E$ of $D_{\operatorname{qc}}(X)$ is pseudocoherent if, and only if, $\mathcal{H}^n(E)$ is coherent for all $n$ and vanishes for $n\gg 0$. So, the subcategory of pseudocoherent objects in $D_{\operatorname{qc}}(X)$ coincides with $D^-_{\operatorname{coh}}(X)$; that is, the full subcategory of complexes with bounded above and coherent cohomology.

\begin{remark}\label{rmk:compact_generator_by_generation}
    Suppose $\mathcal{T}$ is a compactly generated triangulated category\footnote{We do assume that `$\mathcal{T}$ admitting all small coproducts' is part of the data to be compactly generated.}. Denote the collection of compact objects of $\mathcal{T}$ by $\mathcal{T}^c$. Recall an object $G$ of $\mathcal{T}^c$ is called a \textbf{compact generator} for $\mathcal{T}$ if for any object $F$ of $\mathcal{T}$, one has $\operatorname{Hom}(G,F[n]) = 0$ for all integers $n$ if and only if $F$ is the zero object. It is well known that a compact object is a classical generator for $\mathcal{T}^c$ if, and only if, it is a compact generator for $\mathcal{T}$, see for example \cite[\href{https://stacks.math.columbia.edu/tag/09SR}{Tag 09SR}]{StacksProject}.
\end{remark}

\begin{example}\label{ex:perfect_compacts_classical_generator}
    Let $X$ be a quasi-compact quasi-separated scheme. The collection of compact objects in $D_{\operatorname{qc}}(X)$ coincides with $\operatorname{Perf}(X)$. Moreover, $D_{\operatorname{qc}}(X)$ is compactly generated by a single object from $\operatorname{Perf}(X)$. See \cite[Theorem 3.1.1]{Bondal/VandenBergh:2003}.
\end{example}

There are two important Verdier localizations corresponding to an open immersion $r \colon V \to X$ of Noetherian schemes. The first is a Verdier localization $r^\ast \colon D^b_{\operatorname{coh}}(X) \to D^b_{\operatorname{coh}}(V)$ (see \cite[Theorem 4.4]{Elagin/Lunts/Schnurer:2020}). The second is a Bousfield localization sequence
\begin{displaymath}
    D_{\operatorname{Qcoh,X\setminus V}}(X) \to D_{\operatorname{qc}} (X) \xrightarrow{r^\ast} D_{\operatorname{qc}} (V), 
\end{displaymath}
which induces a Verdier localization (up to summands) $r^\ast \colon \operatorname{Perf}(X) \to \operatorname{Perf}(V)$ (see \cite[Theorem 2.1]{Neeman:1992}). Note that these imply that an object of $D_{\operatorname{qc}}(X)$ has bounded coherent cohomology (resp.\ is perfect) if, and only if, its restriction to each affine open satisfies the same property.

\subsection{Singularity category}
\label{sec:prelim_singularity_category}

The \textbf{singularity category} of $X$, denoted $D_{\operatorname{sg}}(X)$, is defined as the Verdier localization of $D^b_{\operatorname{coh}}(X)$ by $\operatorname{Perf}(X)$. This category was first introduced in the algebra setting \cite{Buchweitz/Appendix:2021}, and later rediscovered in the geometric setting by \cite{Orlov:2004}. One can check that $X$ is regular if, and only if, $D_{\operatorname{sg}}(X)$ is trivial. Hueristically, the structure of $D_{\operatorname{sg}}(X)$ `reflects' the singularities of $X$.

We say two Noetherian schemes are \textbf{derived equivalent} if their bounded derived categories of coherent sheaves are triangle equivalent. This is equivalent to detecting triangulated equivalences between their derived category of quasi-coherent sheaves and/or category of perfect complexes (see \cite[Corollary 5.4]{Canonaco/Neeman/Stellari:2024}). 

There has been recent attention to detecting triangulated equivalences between singularity categories, which leads to the notion: We say two Noetherian schemes $X,Y$ are \textbf{singular equivalent} if there is a triangulated equivalence $D_{\operatorname{sg}}(X) \to D_{\operatorname{sg}}(Y)$. This has been detected in various cases. See e.g.\ \cite{Kalck:2021, Knorrer:1987, Matsui:2019, Matsui/Takahashi:2017}. It follows from \cite[Corollary 5.8]{Canonaco/Neeman/Stellari:2024} that derived equivalences imply singular equivalences in our setting.

\subsection{Integral transforms}
\label{sec:prelim_integral_transform}

Let $f_1 \colon Y_1 \to S$ and $f_2 \colon Y_2 \to S$ be morphisms of finite type with $S$ a Noetherian scheme. Consider the fibered square:
\begin{displaymath}
    \begin{tikzcd}[ampersand replacement=\&]
        {Y_1\times_S Y_2} \& {Y_2} \\
        {Y_1} \& S
        \arrow["{p_2}", from=1-1, to=1-2]
        \arrow["{p_1}"', from=1-1, to=2-1]
        \arrow["\lrcorner"{anchor=center, pos=0.125}, draw=none, from=1-1, to=2-2]
        \arrow["{f_2}", from=1-2, to=2-2]
        \arrow["{f_1}"', from=2-1, to=2-2]
    \end{tikzcd}
\end{displaymath} 
The \textbf{integral $S$-transform} associated to an object $K$ in $D_{\operatorname{qc}}(Y_1\times_S Y_2)$ is the functor $\Phi_K$ from $D_{\operatorname{qc}}(Y_1)$ to $D_{\operatorname{qc}}(Y_2)$ given by $\mathbf{R}p_{2,\ast} (\mathbf{L}p_1^\ast (-) \otimes^\mathbf{L} K)$. We will drop the hyphen `$S$-' if it is clear from context. One says $Y_1$ and $Y_2$ are \textbf{Fourier--Mukai $S$-partners} if there is such a $K$ for which $\Phi_K$ yields an equivalence $D^b_{\operatorname{coh}}(Y_1) \to D^b_{\operatorname{coh}}(Y_2)$.

This notion was first introduced in \cite{Mukai:1981} for varieties. These functors have been studied for when they restrict, induce equivalences, or admit adjoints on other subcategories (e.g.\ $D^b_{\operatorname{coh}}(-)$ or $\operatorname{Perf}(-))$. See e.g.\ \cite[\href{https://stacks.math.columbia.edu/tag/0FYP}{Tag 0FYP}]{StacksProject}.

A complete list of references on integral transforms is not reasonable, but we highlight a few for the readers interest: \cite{Orlov:1997, Ballard:2009, Ruiperez/Martin/deSalas:2007, Ruiperez/Martin/deSalas:2009, Rizzardo:2017, Rizzardo/VandenBergh/Neeman:2019, Rizzardo/VandenBergh:2015}. These have also been studied in the setting of algebraic stacks (see \cite[$\S 3$]{Bergh/Schnurer:2020} or \cite{Hall/Priver:2024}). See \cite{Huybrechts:2006,Mukai:1981} for further background.

We record a few cases of interest where singular equivalences can be detected through integral transforms.

\begin{example}
    \hfill
    \begin{enumerate}
        \item Let $X$ be a quasi-projective variety over a field. Suppose $j\colon U \to X$ is an open immersion such that the singular locus of $X$ is contained in $U$. Then $j^\ast \colon D^b_{\operatorname{coh}}(X) \to D^b_{\operatorname{coh}}(U)$ induces a triangulated equivalence $\breve{j}^\ast \colon D_{\operatorname{sg}}(X) \to D_{\operatorname{sg}}(U)$. See \cite[Corollary 2.3]{Chen:2010} or \cite[Proposition 1.14]{Orlov:2004}. Observe that the derived pullback $j^\ast\colon D_{\operatorname{qc}}(X) \to D_{\operatorname{qc}}(U)$ can be realized as the integral transform associated to the graph of $j$.
        \item Let $Y$ be a smooth quasi-projective variety over a field $k$ and $f\colon Y \to \mathbb{A}^1_k$ be a nonzero morphism. Define $g= f+xy \colon Y\times_k \mathbb{A}^2_k \to \mathbb{A}^1_k$, $Z_f := f^{-1}(\{0\})$,  $Z_g := g^{-1} (\{0 \})$, and $W:= Z_f\times_k \{0\} \times_k \mathbb{A}^1_k$ (in $Y\times_k \mathbb{A}^2_k$). Denote by $i\colon W \to Z_g$ the inclusion and $p\colon W\to Z_f$ the flat projection. Then $\mathbf{R}i_\ast p^\ast \colon D_{\operatorname{sg}}(Z_f) \to D_{\operatorname{sg}}(Z_g)$ is a triangulated equivalence, see \cite[Theorem 2.1]{Orlov:2004}. This can be expressed as a composition of integral transforms.
    \end{enumerate}
\end{example}

\section{Preservation}
\label{sec:preservation}

We identify necessary and sufficient conditions under which integral transforms preserve perfect complexes and/or objects with bounded and coherent cohomology. Our work begins with the following important lemma.

\begin{lemma}
    \label{lem:approx_by_proper_lift_functors}
    Let $Y_1$ and $Y_2$ be schemes that are proper over an affine Noetherian scheme $\operatorname{Spec}(R)$. Then, for any exact functor $\Phi\colon \operatorname{Perf}(Y_1) \to \operatorname{Perf}(Y_2)$, there is a unique exact functor $\Phi^\prime \colon D^b_{\operatorname{coh}}(Y_2) \to D^b_{\operatorname{coh}}(Y_1)$, with natural isomorphisms
    \begin{displaymath}
        \operatorname{Hom}(\Phi(A),B)\cong \operatorname{Hom}(A,\Phi^\prime(B))
    \end{displaymath}
    for any $A$ in $\operatorname{Perf}(Y_1)$ and $B$ in $D^b_{\operatorname{coh}}(Y_2)$.
\end{lemma}

\begin{proof}
    It follows, by \cite[Example 0.7]{Neeman:2021b}, that $D^b_{\operatorname{coh}}(Y_i)$ is equivalent to the category of finite cohomological functors $\operatorname{Perf}(Y_i)^{\operatorname{op}} \to \operatorname{Mod}(R)$. Consider a finite cohomological functor $H\colon \operatorname{Perf}(Y_2)^{\operatorname{op}} \to \operatorname{Mod}(R)$. It is easy to see that $H\circ\Phi^{\operatorname{op}} \colon \operatorname{Perf}(Y_1)^{\operatorname{op}}\to \operatorname{Mod}(R)$ is a finite cohomological functor on $\operatorname{Perf}(Y_1)$, which gives us the desired functor $\Phi^\prime\colon D^b_{\operatorname{coh}}(Y_2) \to D^b_{\operatorname{coh}}(Y_1)$. 
\end{proof}

\begin{proposition}
    [cf.\ {\cite[Lemma 3.17]{Ballard:2009}}]
    \label{prop:characterize_perf_for_pseudocoherent_kernel}
    Let $f_1 \colon Y_1 \to S$ and $f_2 \colon Y_2 \to S$ be morphisms of finite type to a Noetherian scheme. Consider the fibered square:
    \begin{displaymath}
        \begin{tikzcd}[ampersand replacement=\&]
            {Y_1\times_S Y_2} \& {Y_2} \\
            {Y_1} \& S
            \arrow["{p_2}", from=1-1, to=1-2]
            \arrow["{p_1}"', from=1-1, to=2-1]
            \arrow["{f_2}", from=1-2, to=2-2]
            \arrow["{f_1}"', from=2-1, to=2-2]
        \end{tikzcd}
    \end{displaymath} 
    Suppose $K$ is an object in $D^-_{\operatorname{coh}}(Y_1 \times_S Y_2)$. Then $\Phi_K (\operatorname{Perf}(Y_1))$ is contained in $\operatorname{Perf}(Y_2)$ if, and only if, $\mathbf{R}p_{2,\ast} (K\otimes^{\mathbf{L}} \operatorname{Perf}(Y_1\times_S Y_2))$ is contained in $\operatorname{Perf}(Y_2)$.
\end{proposition}

\begin{proof}
    Observe the converse direction follows from the fact that $\mathbf{L}p_1^\ast \operatorname{Perf}(Y_1)$ is contained in $\operatorname{Perf}(Y_1\times_S Y_2)$. So we only need to check the forward direction. Assume that $\Phi_K (\operatorname{Perf}(Y_1))$ is contained in $\operatorname{Perf}(Y_2)$. Let $P_i$ be classical generators for $\operatorname{Perf}(Y_i)$. Then, by \cite[Lemma 3.4.1]{Bondal/VandenBergh:2003} coupled with \Cref{rmk:compact_generator_by_generation} and \Cref{ex:perfect_compacts_classical_generator}, one has that $\mathbf{L}p_1^\ast P_1 \otimes^{\mathbf{L}} \mathbf{L}P_2^\ast P_2$ is a classical generator for $\operatorname{Perf}(Y_1 \times_S Y_2)$. Our hypothesis says that $\Phi_K (P_1)$ is an object of $\operatorname{Perf}(Y_2)$. However, by projection formula, we know that 
    \begin{displaymath}
        \Phi_K (P_1) \otimes^{\mathbf{L}} P_2 \cong \mathbf{R}p_{2,\ast} (K \otimes^\mathbf{L} \mathbf{L}p_1^\ast P_1 \otimes^{\mathbf{L}} \mathbf{L}p_2^\ast P_2).
    \end{displaymath}
    Hence, as $(-)\otimes^{\mathbf{L}} P_2$ is an endofunctor on $\operatorname{Perf}(Y_2)$, one has 
    \begin{displaymath}
        \mathbf{R}p_{2,\ast} (K \otimes^\mathbf{L} \mathbf{L}p_1^\ast P_1 \otimes^{\mathbf{L}} \mathbf{L}p_2^\ast P_2)\subseteq \operatorname{Perf}(Y_2). 
    \end{displaymath}
    Then, as $\mathbf{L}p_1^\ast P_1 \otimes^{\mathbf{L}} \mathbf{L}p_2^\ast P_2$ is a classical generator for $\operatorname{Perf}(Y_1\times_S Y_2)$, $\mathbf{R}p_{2,\ast} (K\otimes^{\mathbf{L}} \operatorname{Perf}(Y_1\times_S Y_2))$ is contained in $\operatorname{Perf}(Y_2)$. This completes the proof.
\end{proof}

We record a few cases where \Cref{prop:characterize_perf_for_pseudocoherent_kernel} is applicable.

\begin{example}
    [{\cite[\href{https://stacks.math.columbia.edu/tag/0FYT}{Tag 0FYT}]{StacksProject}}]
    \hfill
    \begin{enumerate}
        \item $K$ can be represented by a bounded complex of coherent $\mathcal{O}_{Y_1\times_S Y_2}$-modules, 
        flat over $Y_2$, with support being proper over $Y_2$;
        \item $f_1$ is proper and flat, and $K$ is in $\operatorname{Perf}(Y_1\times_S Y_2)$;
        \item $f_1$ is proper and flat, $K$ is in $D^b_{\operatorname{coh}}(Y_1\times_S Y_2)$, and for each $s$ in $Y_2$ the derived pullback of $K$ along the projection $Y_1\times_S \operatorname{Spec}(\kappa(s)) \to Y_1 \times_S Y_2$ is locally bounded below (see (\cite[\href{https://stacks.math.columbia.edu/tag/0GEH}{Tag 0GEH}]{StacksProject})).
    \end{enumerate}
\end{example}

\begin{lemma}
    [cf.\ {\cite[Proposition 3.12]{Ballard:2009}}]
    \label{lem:pushforward_preserve_perfect_iff_pullback_coherent_bounded}
    Let $f\colon Y \to X$ be a morphism between schemes which are proper over an affine Noetherian scheme. Then an object $E$ in $D^-_{\operatorname{coh}}(Y)$ satisfies $\mathbf{R}f_\ast (E\otimes^{\mathbf{L}} \operatorname{Perf}(Y))$ being contained in $\operatorname{Perf}(X)$ if, and only if, $E\otimes^{\mathbf{L}} \mathbf{L}f^\ast D^b_{\operatorname{coh}}(X)$ is contained in $D^b_{\operatorname{coh}}(Y)$.
\end{lemma}

\begin{proof}
    First, we show the forward direction. Let $E$ be an object of $D^-_{\operatorname{coh}}(Y)$. Assume $\mathbf{R}f_\ast (E\otimes^{\mathbf{L}} \operatorname{Perf}(Y))$ is contained in $\operatorname{Perf}(X)$. Then, by \cite[Lemma 3.7]{Ballard:2009}, we have $\mathbf{R}f_\ast (\mathbf{R}\operatorname{\mathcal{H}\! \mathit{om}}(E,f^! \mathcal{O}_X)\otimes^{\mathbf{L}} P)$ is an object of $\operatorname{Perf}(X)$ for each $P$ in $\operatorname{Perf}(Y)$. Observe that \cite[Lemma 3.10]{Ballard:2009} tells us $\mathbf{R}f_\ast (\mathbf{R}\operatorname{\mathcal{H}\! \mathit{om}}(E,f^! \mathcal{O}_X)\otimes^{\mathbf{L}} (-))$ is left adjoint to $E \otimes^\mathbf{L} \mathbf{L}f^\ast (-)$ as functors between $D_{\operatorname{qc}}(X)$ and $D_{\operatorname{qc}}(Y)$. Then, by \Cref{lem:approx_by_proper_lift_functors}, one has $E \otimes^\mathbf{L} \mathbf{L}f^\ast (-) \colon D^b_{\operatorname{coh}}(X) \to D_{\operatorname{qc}}(Y)$ must factor through $D^b_{\operatorname{coh}}(Y)$ as desired.

    Next, we check the converse direction. Suppose $E\otimes^{\mathbf{L}} \mathbf{L}f^\ast D^b_{\operatorname{coh}}(X)$ is contained in $D^b_{\operatorname{coh}}(Y)$. Note that $f$ is proper as a morphism between proper schemes over a scheme must itself be proper. This ensures that $\mathbf{R}f_\ast (E\otimes^{\mathbf{L}} P )$ belongs to $D^b_{\operatorname{coh}}(X)$. Now, observe that for any $P$ in $\operatorname{Perf}(Y)$, one has that $E\otimes^{\mathbf{L}} P \otimes^{\mathbf{L}} \mathbf{L}f^\ast D^b_{\operatorname{coh}}(X)$ is also contained in $D^b_{\operatorname{coh}}(Y)$.
    Then, by the projection formula, we see that $\mathbf{R}f_\ast (E\otimes^{\mathbf{L}} P ) \otimes^{\mathbf{L}} D^b_{\operatorname{coh}}(X)$ is also contained in $D^b_{\operatorname{coh}}(X)$ for all $P$ in $\operatorname{Perf}(Y)$. As we already know that $\mathbf{R}f_\ast (E\otimes^{\mathbf{L}} P )$ belongs to $D^b_{\operatorname{coh}}(X)$, it follows from \cite[Theorem 2.3.(3)]{AlonsoTarrio/JeremiasLopez/SanchodeSalas:2023} that $\mathbf{R}f_\ast (E\otimes^{\mathbf{L}} P )$ is contained in $\operatorname{Perf}(X)$ as desired, which completes the proof.
\end{proof}

We now determine when an integral transform preserves objects with bounded and coherent cohomology if the base scheme is affine.

\begin{lemma}\label{lem:integral_transform_bounded_coherent_iff_perf_projection}
    Let $f_1 \colon Y_1 \to S$ and $f_2 \colon Y_2 \to S$ be proper morphisms to an affine Noetherian scheme. Consider the fibered square:
    \begin{displaymath}
        \begin{tikzcd}[ampersand replacement=\&]
            {Y_1\times_S Y_2} \& {Y_2} \\
            {Y_1} \& S.
            \arrow["{p_2}", from=1-1, to=1-2]
            \arrow["{p_1}"', from=1-1, to=2-1]
            \arrow["{f_2}", from=1-2, to=2-2]
            \arrow["{f_1}"', from=2-1, to=2-2]
        \end{tikzcd}
    \end{displaymath}
    Then the following are equivalent for any object $E$ in $D^-_{\operatorname{coh}}(Y_1 \times_S Y_2)$:
    \begin{enumerate}
        \item \label{eq:integral_transform_bounded_coherent_iff_perf_projection1} $\Phi_E (D^b_{\operatorname{coh}}(Y_1))\subseteq D^b_{\operatorname{coh}}(Y_2)$ 
        \item \label{eq:integral_transform_bounded_coherent_iff_perf_projection2} $\mathbf{R}p_{1,\ast} (E\otimes^{\mathbf{L}} \operatorname{Perf}(Y_1\times_S Y_2)) \subseteq \operatorname{Perf}(Y_1)$.
    \end{enumerate}
\end{lemma}

\begin{proof}
    First, we check $\eqref{eq:integral_transform_bounded_coherent_iff_perf_projection1}\implies \eqref{eq:integral_transform_bounded_coherent_iff_perf_projection2}$. Assume $\Phi_E (D^b_{\operatorname{coh}}(Y_1))$ is contained in $D^b_{\operatorname{coh}}(Y_2)$. It suffices, by \Cref{lem:pushforward_preserve_perfect_iff_pullback_coherent_bounded}, to show that $E\otimes^{\mathbf{L}} \mathbf{L}p_1^\ast D^b_{\operatorname{coh}}(Y_1)$ is contained in $D^b_{\operatorname{coh}}(Y_1\times_S Y_2)$.
    
    Note, by \cite[\href{https://stacks.math.columbia.edu/tag/09U7}{Tag 09U7}]{StacksProject}, that $\mathbf{L}p_1^\ast D^b_{\operatorname{coh}}(Y_1)$ is contained in $D^-_{\operatorname{coh}}(Y_1\times_S Y_2)$. Hence, from \cite[\href{https://stacks.math.columbia.edu/tag/09J3}{Tag 09J3}]{StacksProject}, one has $E\otimes^{\mathbf{L}} D^b_{\operatorname{coh}}(Y_1)$ is contained in $D^-_{\operatorname{coh}}(Y_1\times_S Y_2)$. It suffices to check that $E\otimes^\mathbf{L} \mathbf{L}p_1^\ast G$ has bounded cohomology for each $G$ in $D^b_{\operatorname{coh}}(Y_1)$. This can be done by showing each such object belongs to $D^+_{\operatorname{qc}}(Y_1\times_S Y_2)$. 
    
    Let $P_i$ be a compact generator of $D_{\operatorname{qc}}(Y_i)$ for each $i$. Then, by \cite[Lemma 3.4.1]{Bondal/VandenBergh:2003}, one has $\mathbf{L}p_1^\ast P_1 \otimes^{\mathbf{L}} \mathbf{L}p_2^\ast P_2$ is a compact generator for $D_{\operatorname{qc}}(Y_1 \times_S Y_2)$. There is a string of isomorphisms:
    \begin{displaymath}
        \begin{aligned}
            &\operatorname{Ext}^n  (\mathbf{L}p_1^\ast P_1 \otimes^{\mathbf{L}} \mathbf{L}p_2^\ast P_2, E\otimes^\mathbf{L} \mathbf{L}p_1^\ast G)
            \\&\cong \operatorname{Ext}^n (\mathbf{L}p_2^\ast P_2, \mathbf{R}\operatorname{\mathcal{H}\! \mathit{om}}(\mathbf{L}p_1^\ast P_1, E\otimes^\mathbf{L} \mathbf{L}p_1^\ast G))  && \textrm{(\cite[\href{https://stacks.math.columbia.edu/tag/08DH}{Tag 08DH}]{StacksProject})}
            \\&\cong \operatorname{Ext}^n (\mathbf{L}p_2^\ast P_2, \mathbf{R}\operatorname{\mathcal{H}\! \mathit{om}}(\mathbf{L}p_1^\ast P_1,\mathcal{O}_{Y_1\times_S Y_2}) \otimes^{\mathbf{L}} E\otimes^\mathbf{L} \mathbf{L}p_1^\ast G) && \textrm{(\cite[\href{https://stacks.math.columbia.edu/tag/08DQ}{Tag 08DQ}]{StacksProject})}
            \\&\cong \operatorname{Ext}^n (\mathbf{L}p_2^\ast P_2, \mathbf{R}\operatorname{\mathcal{H}\! \mathit{om}}(\mathbf{L}p_1^\ast P_1,\mathbf{L}p_1^\ast \mathcal{O}_{Y_1}) \otimes^{\mathbf{L}} E\otimes^\mathbf{L} \mathbf{L}p_1^\ast G ) && (\mathbf{L}p_1^\ast \mathcal{O}_{Y_1} =\mathcal{O}_{Y_1\times_S Y_2})
            \\&\cong \operatorname{Ext}^n (\mathbf{L}p_2^\ast P_2, \mathbf{L}p_1^\ast (\mathbf{R}\operatorname{\mathcal{H}\! \mathit{om}}( P_1,\mathcal{O}_{Y_1})) \otimes^{\mathbf{L}} E\otimes^\mathbf{L} \mathbf{L}p_1^\ast G ) && \textrm{(\cite[Prop. 22.70]{Gortz/Wedhorn:2023})}
            \\&\cong \operatorname{Ext}^n (\mathbf{L}p_2^\ast P_2, \mathbf{L}p_1^\ast (\mathbf{R}\operatorname{\mathcal{H}\! \mathit{om}}( P_1,\mathcal{O}_{Y_1})\otimes G) \otimes^{\mathbf{L}} E ) && \textrm{(\cite[\href{https://stacks.math.columbia.edu/tag/07A4}{Tag 07A4}]{StacksProject})}
            \\&\cong \operatorname{Ext}^n (P_2, \mathbf{R}p_{2,\ast} (\mathbf{L}p_1^\ast (\mathbf{R}\operatorname{\mathcal{H}\! \mathit{om}}( P_1,\mathcal{O}_{Y_1})\otimes G) \otimes^{\mathbf{L}} E) ) && \textrm{(Adjunction)}
            \\&\cong \operatorname{Ext}^n (P_2, \Phi_E (\mathbf{R}\operatorname{\mathcal{H}\! \mathit{om}}( P_1,\mathcal{O}_{Y_1})\otimes G)).
        \end{aligned}
    \end{displaymath}
    It follows, as $\mathbf{R}\operatorname{\mathcal{H}\! \mathit{om}}( P_1,\mathcal{O}_{Y_1})$ is perfect on $Y_1$, that $\mathbf{R}\operatorname{\mathcal{H}\! \mathit{om}}( P_1,\mathcal{O}_{Y_1})\otimes G$ belongs to $D^b_{\operatorname{coh}}(Y_1)$. Our hypothesis tells us $\Phi_E (\mathbf{R}\operatorname{\mathcal{H}\! \mathit{om}}( P_1,\mathcal{O}_{Y_1})\otimes G)$ belongs to $D^b_{\operatorname{coh}}(Y_2)$. Hence, from \cite[\href{https://stacks.math.columbia.edu/tag/0GEQ}{Tag 0GEQ}]{StacksProject}, one has $\operatorname{Ext}^n (P_2, \Phi_E (\mathbf{R}\operatorname{\mathcal{H}\! \mathit{om}}( P_1,\mathcal{O}_{Y_1})\otimes G)) = 0$ for $0\gg n$. Then, once more from \cite[\href{https://stacks.math.columbia.edu/tag/0GEQ}{Tag 0GEQ}]{StacksProject}, we see that $E\otimes^\mathbf{L} \mathbf{L}p_1^\ast G$ is an object of $D^+_{\operatorname{qc}} (Y_1\times_S Y_2)$ as desired.
    
    Next, we show $\eqref{eq:integral_transform_bounded_coherent_iff_perf_projection2}\implies \eqref{eq:integral_transform_bounded_coherent_iff_perf_projection1}$. It follows, by \Cref{lem:pushforward_preserve_perfect_iff_pullback_coherent_bounded} coupled with the hypothesis, that $E\otimes^{\mathbf{L}} \mathbf{L}p_1^\ast D^b_{\operatorname{coh}}(Y_1)$ is contained in $D^b_{\operatorname{coh}}(Y_1 \times_S Y_2)$. Then, as $p_2$ is proper, we have $\mathbf{R}p_{2,\ast} D^b_{\operatorname{coh}}(Y_1\times_S Y_2)$ is contained in $D^b_{\operatorname{coh}}(Y_2)$. This completes the proof.
\end{proof}

Next, we upgrade ourselves to when the base scheme is not affine.

\begin{proposition}\label{prop:preserves_bounded_coherence_non_affine_base_scheme}
    Let $f_1 \colon Y_1 \to S$ and $f_2 \colon Y_2 \to S$ be proper morphisms to a Noetherian scheme. Consider the fibered square:
    \begin{displaymath}
        \begin{tikzcd}[ampersand replacement=\&]
            {Y_1\times_S Y_2} \& {Y_2} \\
            {Y_1} \& S.
            \arrow["{p_2}", from=1-1, to=1-2]
            \arrow["{p_1}"', from=1-1, to=2-1]
            \arrow["{f_2}", from=1-2, to=2-2]
            \arrow["{f_1}"', from=2-1, to=2-2]
        \end{tikzcd}
    \end{displaymath}
    Then the following are equivalent for any object $K$ in $D^-_{\operatorname{coh}}(Y_1 \times_S Y_2)$:
    \begin{enumerate}
        \item \label{eq:preserves_bounded_coherence_non_affine_base_scheme1} $\Phi_K (D^b_{\operatorname{coh}}(Y_1))\subseteq D^b_{\operatorname{coh}}(Y_2)$,
        \item \label{eq:preserves_bounded_coherence_non_affine_base_scheme2} $\mathbf{R} p_{1,\ast} (K\otimes^{\mathbf{L}} \operatorname{Perf}(Y_1\times_S Y_2))\subseteq \operatorname{Perf}(Y_1)$.
    \end{enumerate}
\end{proposition}

\begin{proof}
    We only check $\eqref{eq:preserves_bounded_coherence_non_affine_base_scheme1} \implies \eqref{eq:preserves_bounded_coherence_non_affine_base_scheme2}$ as the $\eqref{eq:preserves_bounded_coherence_non_affine_base_scheme2} \implies \eqref{eq:preserves_bounded_coherence_non_affine_base_scheme1}$ direction can be argued in a similar fashion. Consider an affine open cover $U_1,\ldots,U_n$ of $S$. This gives us an open cover $Y^\prime_{ij}$ for each $Y_j$. Denote by $s_i \colon U_i \to S$ the associated open immersion of each $U_i$ in $S$. 
     
    There is, for each $i$, a commutative cube:
    \begin{displaymath}
        \begin{tikzcd}[ampersand replacement=\&]
            {Y^\prime_{i1}\times_S Y^\prime_{i2}} \&\& {Y_{i2}^\prime} \\
            \& {Y^\prime_{i1}} \&\&\& U_i \\
            {Y_1 \times_S Y_2} \&\& {Y_2} \\
            \& {Y_1} \&\&\& S.
            \arrow["{q_{i2}}", from=1-1, to=1-3]
            \arrow["{q_{i1}}"', from=1-1, to=2-2]
            \arrow["t_i"', from=1-1, to=3-1]
            \arrow["{g_{i2}}", from=1-3, to=2-5]
            \arrow["{s_{i2}}"'{pos=0.3}, dashed, from=1-3, to=3-3]
            \arrow["{g_{i1}}"', from=2-2, to=2-5]
            \arrow["{s_{i1}}"{pos=0.7}, from=2-2, to=4-2]
            \arrow["s_i"', from=2-5, to=4-5]
            \arrow["{p_2}"{pos=0.3}, dashed, from=3-1, to=3-3]
            \arrow["{p_1}", from=3-1, to=4-2]
            \arrow["{f_2}", from=3-3, to=4-5]
            \arrow["{f_1}", from=4-2, to=4-5]
        \end{tikzcd}
    \end{displaymath}
    whose faces are fibered squares, each vertical edge is an open immersion, and every other edge is a proper morphism. 
    
    Observe, from the cube above, one has the following computation for each object $E$ in $D_{\operatorname{qc}}(Y_1)$:
    \begin{displaymath}
        \begin{aligned}
            s^\ast_{i2} & \Phi_K (E)  \cong s^\ast_{i2} \mathbf{R}p_{2,\ast}  (\mathbf{L}p_1^\ast E  \otimes^{\mathbf{L}} K) 
            \\&\cong \mathbf{R}q_{i2,\ast} t_i^\ast ( \mathbf{L}p_1^\ast E \otimes^{\mathbf{L}} K ) && \textrm{(\cite[Remark 22.94 \& Theorem 22.99]{Gortz/Wedhorn:2023})}
            \\&\cong \mathbf{R}q_{i2,\ast} ( t_i^\ast  \mathbf{L}p_1^\ast E \otimes^{\mathbf{L}} t_i^\ast K ) && \textrm{(\cite[\href{https://stacks.math.columbia.edu/tag/07A4 }{Tag 07A4 }]{StacksProject})}
            \\&\cong \mathbf{R}q_{i2,\ast} ( \mathbf{L}(p_1\circ t_i)^\ast E \otimes^{\mathbf{L}} t_i^\ast K )
            \\&\cong \mathbf{R}q_{i2,\ast} ( \mathbf{L}(s_{i1}\circ q_{i1})^\ast E \otimes^{\mathbf{L}} t_i^\ast K )  && (p_1\circ t_i = s_{i1}\circ q_{i1})
            \\&\cong \mathbf{R}q_{i2,\ast} ( \mathbf{L}q_{i1}^\ast s_{i1}^\ast E \otimes^{\mathbf{L}} t_i^\ast K )
            \\&\cong \Phi_{t_i^\ast K} (s_{i1}^\ast E).
        \end{aligned}
    \end{displaymath}
    Moreover, from \cite[Theorem 4.4]{Elagin/Lunts/Schnurer:2020}, we have a Verdier localization $s_{i1}^\ast \colon D^b_{\operatorname{coh}}(Y_1) \to D^b_{\operatorname{coh}}(Y_{i1}^\prime)$. Hence, if coupled with our hypothesis, it follows that $\Phi_{t_i^\ast K}\colon D^b_{\operatorname{coh}}(Y_{i1}^\prime) \to D_{\operatorname{qc}}(Y_{i2}^\prime)$ factors through $D^b_{\operatorname{coh}}(Y_{i2}^\prime)$ for each $i$. 
        
    Then, by \Cref{lem:integral_transform_bounded_coherent_iff_perf_projection}, one has that $\mathbf{R}q_{{i1},\ast} (t_i^\ast K\otimes^{\mathbf{L}} \operatorname{Perf}(Y^\prime_{i1}\times_S Y^\prime_{i2}))$ is contained in $\operatorname{Perf}(Y^\prime_{i1})$ for each $i$. However, once more from the cube above, we have another computation based on similar reasoning for each $E$ in $D_{\operatorname{qc}}(Y_1 \times_S Y_2)$:
    \begin{displaymath}
        \begin{aligned}
            s^\ast_{i1} \mathbf{R}p_{1,\ast} (K \otimes^{\mathbf{L}} E) 
            &\cong  \mathbf{R}q_{i1,\ast} t_i^\ast (K \otimes^{\mathbf{L}} E) 
            \\&\cong \mathbf{R}q_{i1,\ast} (t_i^\ast K \otimes^{\mathbf{L}} t^\ast_i E).
        \end{aligned}
    \end{displaymath}
    There is, from \cite[Theorem 2.1]{Neeman:1992}, a Verdier localization sequence,
    \begin{displaymath}
        D_{\operatorname{qc},Y_1\setminus Y_{i1}^\prime}(Y_1) \to D_{\operatorname{qc}} (Y_1) \xrightarrow{s_{i1}^\ast} D_{\operatorname{qc}} (Y_{i1}^\prime), 
    \end{displaymath}
    which induces a Verdier localization (up to summands) $s_{i1}^\ast \colon \operatorname{Perf}(Y_1) \to \operatorname{Perf}(Y_{i1}^\prime)$. It follows, for each $P$ in $\operatorname{Perf}(Y_1\times_S Y_2)$ and each $i$, that $s^\ast_{i1} \mathbf{R}p_{1,\ast} (K \otimes^{\mathbf{L}} P)$ is in $\operatorname{Perf}(Y_{i1}^\prime)$. This tells us any such object $\mathbf{R}p_{1,\ast} (K \otimes^{\mathbf{L}} P)$ must belong to $\operatorname{Perf}(Y_1)$ as desired.
\end{proof}

This brings us to our main result for the section. Recall an exact functor $F\colon \mathcal{T}_1 \to \mathcal{T}_2$ between triangulated categories \textit{induces} an exact functor between Verdier localizations $\mathcal{T}_1 / \mathcal{K}_1 \to \mathcal{T}_2 / \mathcal{K}_2$ if $F\colon \mathcal{K}_1 \to \mathcal{T}_2$ factors through $\mathcal{K}_2$.

\begin{theorem}\label{thm:integral_transform_induced_singularity}
    Let $f_1 \colon Y_1 \to S$ and $f_2 \colon Y_2 \to S$ be proper morphisms to a Noetherian scheme. Consider the fibered square:
    \begin{displaymath}
        \begin{tikzcd}[ampersand replacement=\&]
            {Y_1\times_S Y_2} \& {Y_2} \\
            {Y_1} \& S.
            \arrow["{p_2}", from=1-1, to=1-2]
            \arrow["{p_1}"', from=1-1, to=2-1]
            \arrow["{f_2}", from=1-2, to=2-2]
            \arrow["{f_1}"', from=2-1, to=2-2]
        \end{tikzcd}
    \end{displaymath}
    Then the following are equivalent for any object $K$ in $D^-_{\operatorname{coh}}(Y_1 \times_S Y_2)$:
    \begin{enumerate}
        \item $\Phi_K$ induces an exact functor $\breve{\Phi_K}\colon D_{\operatorname{sg}}(Y_1) \to D_{\operatorname{sg}}(Y_2)$
        \item $\mathbf{R}p_{i,\ast} (K\otimes^{\mathbf{L}} \operatorname{Perf}(Y_1\times_S Y_2))\subseteq \operatorname{Perf}(Y_i)$ for each $i$.
    \end{enumerate}
\end{theorem}

\begin{proof}
    This follows from \Cref{prop:characterize_perf_for_pseudocoherent_kernel,prop:preserves_bounded_coherence_non_affine_base_scheme}
\end{proof}

\section{Some consequences}
\label{sec:some_consequences}

This section highlights a few examples and applications to our work in \Cref{sec:preservation}. To start, we provide an instance where an integral transform does not induce an exact functor between singularity categories.

\begin{example}
    Let $s_i\colon Y_i \to \operatorname{Spec}(k)$ be proper varieties over a field $k$ for $i=1,2$. Consider the fibered square:
    \begin{displaymath}
        \begin{tikzcd}[ampersand replacement=\&]
            {Y_1\times_S Y_2} \& {Y_2} \\
            {Y_1} \& \operatorname{Spec}(k).
            \arrow["{p_2}", from=1-1, to=1-2]
            \arrow["{p_1}"', from=1-1, to=2-1]
            \arrow["{s_2}", from=1-2, to=2-2]
            \arrow["{s_1}"', from=2-1, to=2-2]
        \end{tikzcd}
    \end{displaymath}
    Suppose $A_1$ is an object of $D^b_{\operatorname{coh}}(Y_1)$ and $A_2$ is an of object $D^b_{\operatorname{coh}}(Y_2)$ which is not in $\operatorname{Perf}(Y_2)$. Assume that $\mathbf{R}s_{1,\ast} (A_1 \otimes^\mathbf{L} D^b_{\operatorname{coh}}(Y_1))$ is contained in $D^b_{\operatorname{coh}}(\operatorname{Spec}(k))$ (e.g.\ $A_1$ is in $\operatorname{Perf}(Y_1)$). Then a direct computation shows the following:
    \begin{displaymath}
        \begin{aligned}
            \Phi_{p^\ast_1 A_1 \otimes^{\mathbf{L}} p_2^\ast A_2} (E) &:= \mathbf{R}p_{2,\ast} (p_1^\ast E \otimes^{\mathbf{L}} p^\ast_1 A_1 \otimes^{\mathbf{L}} p_2^\ast A_2 )
            \\&\cong \mathbf{R}p_{2,\ast} (p_1^\ast E \otimes^{\mathbf{L}} p^\ast_1 A_1) \otimes^{\mathbf{L}} A_2 && \textrm{(\cite[\href{https://stacks.math.columbia.edu/tag/08EU}{Tag 08EU}]{StacksProject})}
            \\&\cong \mathbf{R}p_{2,\ast} p_1^\ast (E \otimes^{\mathbf{L}} A_1) \otimes^{\mathbf{L}} A_2 && \textrm{(\cite[\href{https://stacks.math.columbia.edu/tag/07A4}{Tag 07A4}]{StacksProject})}
            \\&\cong s^\ast_2 \mathbf{R}s_{1,\ast} (E \otimes^{\mathbf{L}} A_1) \otimes^{\mathbf{L}} A_2 && \textrm{(\cite[Thm. 22.99]{Gortz/Wedhorn:2023})}
            \\&\cong (\bigoplus_{n\in \mathbb{Z}} s^\ast_2 \mathcal{O}_{\operatorname{Spec}(k)}^{\oplus r_n}[n]) \otimes^{\mathbf{L}} A_2
            \\&\cong (\bigoplus_{n\in \mathbb{Z}} \mathcal{O}_{Y_2}^{\oplus r_n}[n]) \otimes^{\mathbf{L}} A_2
            \\&\cong \bigoplus_{n\in \mathbb{Z}}A_2^{\oplus r_n}[n].
        \end{aligned}
    \end{displaymath}
    Observe, from our hypothesis on $A_1$, that $r_n\not=0$ for at most finitely\footnote{Any object of $D^b_{\operatorname{coh}}(\operatorname{Spec}(k))$ is isomorphic to an object of the form $\oplus_{t\in \mathbb{Z}} \mathcal{O}_{\operatorname{Spec}(k)}^{\oplus d_t}[t]$.} many $n$. This tells us, as $A_2$ is not in $\operatorname{Perf}(Y_2)$, 
    \begin{displaymath}
        \Phi_{p^\ast_1 A_1 \otimes^{\mathbf{L}} p_2^\ast A_2} \colon D^b_{\operatorname{coh}}(Y_1) \to D_{\operatorname{qc}}(Y_2)
    \end{displaymath}
    factors through $D^b_{\operatorname{coh}}(Y_2)$. However, 
    \begin{displaymath}
        \Phi_{p^\ast_1 A_1 \otimes^{\mathbf{L}} p_2^\ast A_2} \colon \operatorname{Perf}(Y_1) \to D_{\operatorname{qc}}(Y_2)
    \end{displaymath}
    cannot factor through $\operatorname{Perf}(Y_2)$. 
\end{example}

On the other hand, there are classical instances where \Cref{thm:integral_transform_induced_singularity} is applicable. 

\begin{example}\label{ex:bergh_schnurer_stacks}
    Let $X$ be a variety over a field.
    \begin{enumerate}
        \item Consider the projectivization $p\colon \mathbb{P}_X(\mathcal{E}) \to X$ of finite locally free sheaf $\mathcal{E}$ on $X$. Then, by \cite[Theorem 6.7]{Bergh/Schnurer:2020} functors $\Phi_n \colon E\mapsto p^\ast E \otimes^{\mathbf{L}} \mathcal{O}_{\mathbb{P}_X(\mathcal{E})} (n)$ are integral transforms (with kernel $\mathcal{O}_{\mathbb{P}_X(\mathcal{E})} (n)$) for each integer $n$. Moreover, \cite[Corollary 6.8]{Bergh/Schnurer:2020} ensures each $\Phi_n$ preserves both perfect complexes and those with bounded coherent cohomology.
        \item Suppose $i\colon Z \to X$ is a closed immersion that is regular (in the sense of \cite[\href{https://stacks.math.columbia.edu/tag/0638}{Tag 0638}]{StacksProject}) of constant codimension $c\geq 0$. Denote by $f\colon \widetilde{X}\to X$ for the blowup of $X$ along $Z$. Consider the following fibered square:
        \begin{displaymath}
            \begin{tikzcd}
                E & {\widetilde{X}} \\
                Z & X
                \arrow["{i^\prime}", from=1-1, to=1-2]
                \arrow["{f^\prime}"', from=1-1, to=2-1]
                \arrow["f", from=1-2, to=2-2]
                \arrow["i"', from=2-1, to=2-2]
            \end{tikzcd}
        \end{displaymath}
        where $E$ is the exceptional divisor. Then, by \cite[Theorem 6.9]{Bergh/Schnurer:2020}, the functors $\Phi_j\colon A \mapsto \mathcal{O}_{\widetilde{X}}(-j \cdot E) \otimes^{\mathbf{L}} \mathbf{R}i^\prime_\ast \mathbf{L}(f^\prime)^\ast A$ are integral transforms if $j\leq 0$. Moreover, from \cite[Corollary 6.10]{Bergh/Schnurer:2020}, if $-c+1 \leq j \leq 0$, one has that $\Phi_j$ preserves perfect complexes and objects with bounded coherent cohomology.
    \end{enumerate}
\end{example}

The following result is a (natural) generalization of \cite[Lemma's 3.22 \& 3.23]{Ballard:2009}. Before doing so, we remind ourselves of two facts regarding adjoints for triangulated categories.

\begin{itemize}
    \item Let $F\colon \mathcal{T}\to \mathcal{S}$ be an exact functor between triangulated categories. Assume $\mathcal{T}$ is compactly generated. Then $F$ admits a right adjoint if, and only if, it preserves small coproducts. See \cite[Theorem 4.1]{Neeman:1996}.
    \item Let $F\colon \mathcal{S} \rightleftarrows \mathcal{T} \colon G$ be an adjoint pair of exact functors between triangulated categories. Assume $\mathcal{S}$ is compactly generated. Then $F$ preserves compacts if, and only if, $G$ preserves small coproducts. See \cite[Theorem 5.1]{Neeman:1996}.
\end{itemize}

\begin{proposition}
    \label{prop:integral_transform_on_dbcoh_iff_left_adjoint}
    With the notation of \Cref{thm:integral_transform_induced_singularity}, assume additionally that $S$ is affine.
    Then $\Phi_K\colon D_{\operatorname{qc}}(Y_1) \to D_{\operatorname{qc}}(Y_2)$ admits a left adjoint if, and only if, the functor $\Phi_K \colon D^b_{\operatorname{coh}}(Y_1) \to D_{\operatorname{qc}}(Y_2)$ factors through $D^b_{\operatorname{coh}}(Y_2)$. In such a situation, $\Phi_{K^\prime}$ is the left adjoint to $\Phi_K$ on $D_{\operatorname{qc}}$, 
    where $K^\prime:= \mathbf{R} \operatorname{\mathcal{H}\! \mathit{om}} (K, p^!_1 \mathcal{O}_{Y_1})$. 
\end{proposition}

\begin{proof}
    First, we prove the forward direction where $\Phi_K$ admits a left adjoint $\Phi$ as a functor on $D_{\operatorname{qc}}$. Observe that $\Phi_K$ always admits a right adjoint as a functor on $D_{\operatorname{qc}}$ that is given by
    \begin{displaymath}
        E \mapsto \mathbf{R} p_{1,\ast} \mathbf{R} \operatorname{\mathcal{H}\! \mathit{om}}(K,\mathbf{L}p_2^\ast E).
    \end{displaymath}
    This means $\Phi_K$ preserves small coproducts, and so, $\Phi$ preserves compacts. In other words, $\Phi(\operatorname{Perf}(Y_2))\subseteq \operatorname{Perf}(Y_1)$. Then \Cref{lem:approx_by_proper_lift_functors} tells us $\Phi_K (D^b_{\operatorname{coh}}(Y_1))\subseteq D^b_{\operatorname{coh}}(Y_2)$ as desired, e.g.\ the restriction of $\Phi_K$ to $D^b_{\operatorname{coh}}(Y_1)$ agrees with the unique functor in \Cref{lem:approx_by_proper_lift_functors} (indeed, \cite[Example 0.7]{Neeman:2021c} gives a triangulated equivalence between the category of finite cohomological functors on perfect complexes and the bounded derived category of bounded pseudocoherent complexes).

    Next, we check the converse direction. It follows from \Cref{lem:integral_transform_bounded_coherent_iff_perf_projection} that $\mathbf{R}p_{1,\ast} (K \otimes^{\mathbf{L}} \operatorname{Perf}(Y_1\times_S Y_2))$ is contained in $\operatorname{Perf}(Y_1)$. There is a string of natural isomorphisms:
    \begin{displaymath}
        \begin{aligned}
            \operatorname{Hom}& (E,\Phi_K (G))
            \cong \operatorname{Hom}( E ,\mathbf{R}p_{2,\ast} (K \otimes^{\mathbf{L}} \mathbf{L}p_1^\ast G)) && \textrm{(Definition)}
            \\&\cong \operatorname{Hom}( \mathbf{L}p_2^\ast E , K \otimes^{\mathbf{L}} \mathbf{L}p_1^\ast G) && \textrm{(Adjunction)}
            \\&\cong \operatorname{Hom}( \mathbf{R}p_{1,\ast} ( \mathbf{R} \operatorname{\mathcal{H}\! \mathit{om}} (K, p^!_1 \mathcal{O}_{Y_1}) \otimes^{\mathbf{L}} \mathbf{L}p_2^\ast E ) , G) && (\textrm{\cite[Lemma 3.10]{Ballard:2009} with } p_1)
        \end{aligned}
    \end{displaymath}
    This shows the desired claim.
\end{proof}

We provide a strengthening of \cite[Proposition 9]{Olander:2023} to the case of varieties with mild singularities. Recall that a variety $X$ over a field $k$ is called a \textit{(birational) derived splinter} if the natural map $\mathcal{O}_X \to \mathbf{R}f_\ast \mathcal{O}_Y$ splits for every proper (resp.\ birational) surjective morphism $f\colon Y \to X$. The notion of derived splinters was introduced by Bhatt \cite{Bhatt:2012}, while the birational variant originates in unpublished work of Kov\'{a}cs. Over characteristic zero the two notions coincide with having rational singularities (see \cite{Lank:2025}), but in positive characteristic birational derived splinters need not be derived splinters (see e.g.\ \cite[Example 4.16]{Lank:2025}).

\begin{proposition}
    \label{prop:extending_olander_splinters}
    Let $Y_1$ and $Y_2$ be birational derived splinters that are proper over an uncountable field $k$. Consider the following situation:
    \begin{itemize}
        \item for each $i$ there is a resolution of singularities $\pi_i \colon \widetilde{Y}_i \to Y_i$
        \item there is $K\in D^-_{\operatorname{coh}}(Y_1\times_k Y_2)$ such that $\Phi_K\colon D^b_{\operatorname{coh}}(Y_1) \to D^b_{\operatorname{coh}}(Y_2)$ is fully faithful.
    \end{itemize}
    If $\Phi_K(\operatorname{Perf}(Y_1)) \subseteq \operatorname{Perf}(Y_2)$, then $\dim Y_1 \leq \dim Y_2$.
\end{proposition}

\begin{proof}
    First, one can augment the proof \cite[Proposition 6.7]{DeDeyn/Lank/ManaliRahul:2024b} for birational derived splinters as the same splitting condition holds on the unit morphisms. Now, if coupled with \cite[Lemma 7]{Olander:2023}, we see that $\operatorname{Cdim} D^b_{\operatorname{coh}}(Y_i) = \dim  Y_i$ for each $i$. Moreover, $\Phi_K$ always has a right adjoint $\Phi$ (see proof of \Cref{prop:integral_transform_on_dbcoh_iff_left_adjoint}). Then \Cref{lem:approx_by_proper_lift_functors}, coupled with our hypothesis, ensures $\Phi$ restricts to give an exact functor $D^b_{\operatorname{coh}}(Y_2) \to D^b_{\operatorname{coh}}(Y_1)$. However, $\Phi_K$ being fully faithful ensures $\Phi \circ \Phi_K \to \mathbf{1}$ is an isomorphism (see e.g.\ \cite[\href{https://stacks.math.columbia.edu/tag/07RB}{Tag 07RB}]{StacksProject}), which in turn implies that $D^b_{\operatorname{coh}}(Y_2) \to D^b_{\operatorname{coh}}(Y_1)$ is a Verdier localization. This immediately implies that $\operatorname{Cdim}D^b_{\operatorname{coh}}(Y_1) \leq  \operatorname{Cdim}D^b_{\operatorname{coh}}(Y_2)$, and hence $\dim Y_1 \leq \dim Y_2$.
\end{proof}

\begin{example}\label{ex:rational_singularities}
    In characteristic zero, being a (birational) derived splinter is equivalent to having rational singularities; see \cite{Kovacs:2000, Bhatt:2012,Lank:2025} for details. Moreover, \cite{Hironaka:1964a,Hironaka:1964b} ensures the existence of resolution of singularities in characteristic zero. 
\end{example}

\begin{proposition}\label{prop:obstruction_to_singular_equivalence}
    Let $f_1\colon Y_1 \to S$ and $f_2\colon Y_2 \to S$ be proper morphisms to a quasi-compact regular scheme where at least one such morphism is flat. Suppose $Y_1$ and $Y_2$ are Fourier--Mukai $S$-partners given by a kernel $K$ in $D^b_{\operatorname{coh}}(Y_1\times_S Y_2)$. If $Y_1$ or $Y_2$ is not regular, then $K\not\in\operatorname{Perf}(Y_1\times_S Y_2)$.
\end{proposition}

\begin{proof}
    We prove the claim by contradiction. That is, $K\in \operatorname{Perf}(Y_1\times_S Y_2)$. The hypothesis is that $\Phi_K$ restricts to a triangulated equivalence $D^b_{\operatorname{coh}}(Y_1) \to D^b_{\operatorname{coh}}(Y_2)$. However, as one of the $Y_i$ is not regular, \cite[Corollary 5.9]{Canonaco/Neeman/Stellari:2024} tells us the other $Y_j$ cannot be regular. Let $G_i$ be a classical generator for $\operatorname{Perf}(Y_i)$. Denote by $\pi_i \colon Y_1 \times_S Y_2 \to Y_i$ the projection morphisms. Then \cite[Lemma 3.4.1]{Bondal/VandenBergh:2003} tells us $\mathbf{L} \pi_1^\ast G_1 \otimes^{\mathbf{L}} \mathbf{L} \pi_2^\ast G_2$ is a classical generator for $\operatorname{Perf}(Y_1 \times_S Y_2)$. Choose an object $E$ in $D^b_{\operatorname{coh}}(Y_2)$. There is $E^\prime$ in $D^b_{\operatorname{coh}}(Y_1)$ such that $\mathbf{R}\pi_{2,\ast} (\mathbf{L} \pi^\ast_1 E^\prime \otimes^{\mathbf{L}} K)\cong E$. Note that $K$ being perfect means $K$ is finitely built by $\mathbf{L} \pi_1^\ast G_1 \otimes^{\mathbf{L}} \mathbf{L} \pi_2^\ast G_2$. Observe that projection formula tells us
    \begin{displaymath}
        G_2 \otimes^{\mathbf{L}} E \cong \mathbf{R}\pi_{2,\ast} (\mathbf{L} \pi^\ast_1 E^\prime \otimes^{\mathbf{L}} K \otimes^{\mathbf{L}} \mathbf{L} \pi_2^\ast G_2).
    \end{displaymath}
    Clearly, $\mathbf{L} \pi^\ast_1 E^\prime \otimes^{\mathbf{L}} K \otimes^{\mathbf{L}} \mathbf{L} \pi_2^\ast G_2$ is finitely built by $\mathbf{L} \pi^\ast_1 E^\prime \otimes^{\mathbf{L}} \mathbf{L} \pi_1^\ast G_1 \otimes^{\mathbf{L}} \mathbf{L} \pi_2^\ast G_2$. Hence, we have that 
    \begin{displaymath}
        \begin{aligned}
            G_2 \otimes^{\mathbf{L}} E &\cong \mathbf{R}\pi_{2,\ast}( \mathbf{L} \pi^\ast_1 E^\prime \otimes^{\mathbf{L}} K \otimes^{\mathbf{L}} \mathbf{L} \pi_2^\ast G_2) \\& \in \langle\mathbf{R}\pi_{2,\ast} (\mathbf{L} \pi^\ast_1 E^\prime \otimes^{\mathbf{L}} \mathbf{L} \pi_1^\ast G_1 \otimes^{\mathbf{L}} \mathbf{L} \pi_2^\ast G_2) \rangle.
        \end{aligned}
    \end{displaymath}
    We have an isomorphism
    \begin{displaymath}
        \mathbf{R}\pi_{2,\ast} (\mathbf{L} \pi^\ast_1 E^\prime \otimes^{\mathbf{L}} \mathbf{L} \pi_1^\ast G_1 \otimes^{\mathbf{L}} \mathbf{L} \pi_2^\ast G_2)\cong \mathbf{R}\pi_{2,\ast} (\mathbf{L} \pi^\ast_1 E^\prime \otimes^{\mathbf{L}} \mathbf{L} \pi_1^\ast G_1 ) \otimes^{\mathbf{L}} G_2.
    \end{displaymath}
    Since $S$ is regular, $D^b_{\operatorname{coh}}(S)= \operatorname{Perf}(S)$. Let $G$ be a classical generator for $\operatorname{Perf}(S)$. By flat base change, 
    \begin{displaymath}
        \mathbf{R}\pi_{2,\ast} (\mathbf{L} \pi^\ast_1 E^\prime \otimes^{\mathbf{L}} \mathbf{L} \pi_1^\ast G_1 ) \in \langle \mathbf{L} f_2^\ast G\rangle.
    \end{displaymath}
    However, $\mathbf{L} f_2^\ast G\in \operatorname{Perf}(Y_2) = \langle G_2 \rangle$. This tells us that $G_2 \otimes^{\mathbf{L}} E$ is finitely built by $G_2$; that is, $G_2 \otimes^{\mathbf{L}} E$ is perfect. But this is absurd as it implies $E \in \operatorname{Perf}(Y_2)$. To be precise, we have shown $\Phi_K (D^b_{\operatorname{coh}}(Y_1))\subseteq \operatorname{Perf}(Y_2)$, and yet $\Phi_K$ restricts to give a triangulated equivalence $D^b_{\operatorname{coh}}(Y_1) \to D^b_{\operatorname{coh}}(Y_2)$. So being that $Y_2$ is not regular, we would obtain from the work above that $\operatorname{Perf}(Y_2) = D^b_{\operatorname{coh}}(Y_2)$, which is a contradiction.
\end{proof}

\bibliographystyle{alpha}
\bibliography{mainbib}

\end{document}